\documentclass{article}

\usepackage[english]{babel}
\usepackage{amsthm}
\usepackage{amsmath}
\usepackage{amssymb}

\usepackage{graphicx}
\usepackage{enumitem}
\usepackage[colorlinks]{hyperref}
\usepackage{caption}
\usepackage{subcaption}
\usepackage{pict2e,color,colordvi,amscd}
\usepackage{indentfirst}
\usepackage{tikz}
\usepackage{float}
\usetikzlibrary{calc}

\newtheorem{definition}{Definition}[section]
\newtheorem{theorem}[definition]{Theorem}
\newtheorem{lemma}[definition]{Lemma}

\newtheorem{corollary}[definition]{Corollary}

\newtheorem{conjecture}{Conjecture}

\newtheorem{question}{Question}
\newtheorem{observation}{Observation}

\theoremstyle{definition}
\newtheorem{example}{Example}
\newtheorem{remark}{Remark}

\theoremstyle{plain}

\setlist[itemize,enumerate]{itemsep=1pt, topsep=2pt, parsep=0pt}
\newcommand{\Z}{\mathbb Z}
\newcommand{\supp}{\operatorname{supp}}

\title{On the Spectra of Chromatic Number and Chromatic Index of Cyclic Covers}

\author{
Guantao Chen\thanks{Department of Mathematics and Statistics, Georgia State University, Atlanta, GA 30303, USA.
Email addresses: \texttt{\{gchen, hvanderholst, yma29\}@gsu.edu}.}
\and
Hein van der Holst\footnotemark[1]
\and
Rong Luo\thanks{Department of Mathematics, West Virginia University, Morgantown, WV 26506-6310, USA. 
Email address: rluo@mail.wvu.edu.}
\and 
Yuying Ma\footnotemark[1]
}

\begin{document}
\date{\relax}
\maketitle
\begin{abstract}
For a fixed integer $\ell \ge 2$, we study what values of chromatic index and chromatic number can be attained by some $\ell$-fold cyclic cover of a loopless multigraph. For edge-coloring, we first investigate the density, a fundamental lower bound for the chromatic index, and show that the density of every $\ell$-fold cyclic cover of a graph $G$ is at most that of $G$. We further prove that if $\ell$ is even, then the spectrum of chromatic indices
over all $\ell$-fold cyclic covers of $G$ contains every integer between $\Delta(G)$ and $\chi'(G)$. When $\ell$ is odd, the chromatic-index spectrum need not be complete in general; for edge-chromatic critical graphs, we determine exactly which values are attainable. For vertex-coloring, we prove that if $\chi(G)\ge 3$, then the spectrum of chromatic numbers over all $\ell$-fold cyclic covers of $G$ contains every integer between $3$ and $\chi(G)$. Moreover, this spectrum contains $2$ if and only if $G$ is bipartite or $\ell$ is even.

\vspace{0.3cm}
\end{abstract}

\section{Introduction}
All graphs in this paper are finite and loopless, but may have parallel edges, unless otherwise specified. Let $G$ be a graph with vertex set $V(G)$ and edge set $E(G)$. We call $|G|:=|V(G)|$ the \emph{order} of $G$, and $|E(G)|$ the \emph{size} of $G$. We write $\delta(G)$ and $\Delta(G)$ for its \emph{minimum degree} and \emph{maximum degree}, respectively. 
Fix an arbitrary reference orientation of the edges of $G$. For any $e \in E(G)$, let $t(e)$ and $h(e)$ denote the \emph{tail} and \emph{head} of $e$. 
For an integer $\ell \ge 1$, let $\Z_\ell$ be the additive cyclic group of residues modulo $\ell$. A \emph{$\Z_\ell$-voltage assignment} on $G$ is a function $\alpha: E(G) \to \Z_\ell$. When an edge $e$ is considered with the opposite orientation, its voltage is defined to be $-\alpha(e)$. The \emph{support} of $\alpha$ is $\supp(\alpha):=\{e\in E(G) \,\mid\, \alpha(e)\ne 0\}$. 

The \emph{$\ell$-fold cyclic cover} of $G$ derived from a $\Z_\ell$-voltage assignment $\alpha$, denoted by $\widetilde G := \widetilde G(\alpha,\ell)$, is a graph with vertex set $V(\widetilde G) = V(G) \times \Z_\ell$ and edge set $E(\widetilde G) = E(G) \times \Z_\ell$, where, for each $e \in E(G)$ and $i \in \Z_\ell$, the edge $(e,i)$ joins $(t(e),i)$ and $(h(e),i+\alpha(e))$.
We refer to $G$ as the \emph{base graph} of $\widetilde G$. By definition, each vertex $(v,i)$ of $\widetilde G$ has the same degree as $v$ in $G$. Consequently, $\delta(\widetilde G) = \delta(G)$ and $\Delta(\widetilde G) = \Delta(G)$.
Although the construction uses a reference orientation, the resulting cyclic cover is independent of this choice. Therefore, throughout the paper, we fix one arbitrary reference orientation and identify each edge of $G$ with its chosen orientation whenever no confusion arises.

Graph coverings provide a standard way for constructing larger graphs from smaller base graphs while preserving local structure. Voltage graphs, introduced by Gross~\cite{Gross1974} and developed by Gross and Tucker~\cite{GrossTucker1977, GrossTucker1987}, give an algebraic description of graph coverings by assigning group elements to the edges of a base graph. Here we focus on voltage assignments with values in $\Z_\ell$, which produce the $\ell$-fold cyclic covers studied in this paper. The case $\ell = 2$ corresponds to the well-studied setting of double covers of signed graphs. Related work can be found in~\cite{FengKutnarMalnicMarusic2008, Hofmeister1988, Waller1976}.

The chromatic number and chromatic index of graph coverings have been studied in several settings. Amit, Linial, and Matou\v{s}ek~\cite{ALM2002} obtained asymptotic bounds on the chromatic number of random $n$-fold covers of a fixed base graph. Kim and Lee~\cite{KIM2008} introduced a relative chromatic number associated with a graph and one of its spanning subgraphs, and showed that it equals the chromatic number of the corresponding double cover. For edge-coloring, Pisanski, Shawe-Taylor, and Vrabec~\cite{PISANSKI1983} gave sufficient conditions under which a graph bundle has chromatic index equal to its maximum degree. Their framework includes cyclic covers as a special case. Later, Plachta~\cite{plachta2020} studied conditions under which a finite cover of a snark remains a snark. Here we focus on $\ell$-fold cyclic covers of a fixed base graph and investigate how their chromatic number and chromatic index depend on the choice of voltage assignment.

Let $G$ be a graph. A set of pairwise nonadjacent vertices is called an \emph{independent set}, while a set of pairwise nonadjacent edges is called a \emph{matching}. A \emph{proper $k$-vertex-coloring} of $G$ assigns one of $k$ colors to each vertex so that adjacent vertices receive distinct colors. In other words, it partitions $V(G)$ into at most $k$ independent sets. The \emph{chromatic number} of $G$, denoted by $\chi(G)$, is the least integer $k$ for which $G$ admits a proper $k$-vertex-coloring. Similarly, a \emph{proper $k$-edge-coloring} of $G$ assigns one of $k$ colors to each edge so that adjacent edges receive distinct colors. It partitions $E(G)$ into at most $k$ matchings. The \emph{chromatic index} of $G$, denoted by $\chi'(G)$, is the least integer $k$ for which $G$ admits a proper $k$-edge-coloring.

Let $\widetilde G(\alpha,\ell)$ be the cyclic
cover of a graph $G$ derived by an arbitrary $\Z_\ell$-voltage assignment $\alpha$. By definition, $V(\widetilde G(\alpha,\ell))= V(G)\times\Z_\ell$ and $E(\widetilde G(\alpha,\ell)) = E(G)\times\Z_\ell$. For any proper $t$-vertex-coloring of $G$ with color classes $V(G) = I_1\dot\cup \cdots \dot\cup I_t$, the partition $V(\widetilde G(\alpha,\ell)) = (I_1\times \Z_\ell) \dot\cup \cdots\dot\cup (I_t\times \Z_\ell)$ gives a proper $t$-vertex-coloring of $\widetilde G(\alpha,\ell)$. Likewise, for any proper $k$-edge-coloring of $G$ with color classes $E(G) = M_1\dot\cup \cdots \dot\cup M_k$, the partition $E(\widetilde G(\alpha,\ell)) = (M_1\times \Z_\ell)\dot\cup \cdots \dot\cup (M_k\times \Z_\ell)$ gives a proper $k$-edge-coloring of $\widetilde G(\alpha,\ell)$. Together with $\Delta(\widetilde G(\alpha,\ell))= \Delta(G)$, this gives
\[
\chi(\widetilde G(\alpha,\ell))\le \chi(G)
\quad\text{and}\quad
\Delta(G)\le \chi'(\widetilde G(\alpha,\ell))\le \chi'(G).
\]

Fix an integer $\ell \ge 2$ and a graph $G$. This leads to the following
natural questions.
\begin{question}\label{q1}
For which $s \in \{\Delta(G), \Delta(G)+1,\dots, \chi'(G)\}$ does there exist a $\Z_\ell$-voltage assignment $\alpha$ such that $\chi'(\widetilde G(\alpha,\ell)) = s$?
\end{question}
\begin{question}\label{q2}
For which $t \in \{2,\dots,\chi(G)\}$ does there exist a $\Z_\ell$-voltage assignment $\alpha$ such that $\chi(\widetilde G(\alpha,\ell)) = t$? 
\end{question}

The following result shows that, when $\ell$ is even, every value of $s$ in Question~\ref{q1} is attained.
\begin{theorem}\label{thm:edge-coloring-brief}
Let $G$ be a graph. If $\ell$ is even, then for every $s \in \{\Delta(G), \dots, \chi'(G)\}$, there exists a $\Z_\ell$-voltage assignment $\alpha$ such that $\chi'(\widetilde G(\alpha, \ell)) = s$. Moreover, if $s\ne\chi'(G)$, then $\alpha$ may be chosen so that $\chi'(\widetilde G(\alpha', \ell)) > \chi'(\widetilde G(\alpha, \ell))$ for every $\Z_\ell$-voltage assignment $\alpha'$ with $\supp(\alpha') \subsetneq \supp(\alpha)$.
\end{theorem}

For edge-coloring, the maximum degree gives the immediate lower bound $\chi'(G) \ge \Delta(G)$. Vizing's theorem~\cite{Vizing1964} states that $\chi'(G) \le \Delta(G)+\mu(G)$, where $\mu(G)$ denotes the maximum edge multiplicity of $G$. In particular, if $G$ is simple, then $\chi'(G) \le \Delta(G)+1$. For multigraphs, another lower bound for the chromatic index arises from dense subgraphs. Define the \emph{density} of $G$ by
$$\Gamma(G):=\max_{H\subseteq G,\ |H|\ge2}\left\lceil\frac{|E(H)|}{\lfloor |H|/2\rfloor}\right\rceil.$$ 
Since each color class contains at most $\lfloor |H|/2\rfloor$ edges in any subgraph $H\subseteq G$, we have $\chi'(G)\ge \Gamma(G)$, and hence $\chi'(G) \ge \max\{\Delta(G), \Gamma(G)\}$.  The Goldberg-Seymour Conjecture~\cite{Goldberg1973,Seymour1979}, recently confirmed by Chen, Hao, Jing, Yu, and Zang~\cite{CHYZ2024,CJZ2025,J2026}, states that if $\chi'(G) \ge \Delta(G) + 2$, then $\chi'(G) = \Gamma(G)$. Equivalently, $\chi'(G) \le \max\{\Delta(G)+1, \Gamma(G)\}$. We show that every $\ell$-fold cyclic cover of a graph $G$ has density at most $\Gamma(G)$; see Theorem~\ref{thm:density}.

Theorem~\ref{thm:edge-coloring-brief} implies that when $\ell$ is even, there exists a $\Z_\ell$-voltage assignment $\alpha$ such that $\Gamma(\widetilde G(\alpha,\ell)) \le \Delta(G)$. On the other hand, if both $\ell$ and $|G|$ are odd, then every $\ell$-fold cyclic cover $\widetilde G(\alpha,\ell)$ of $G$ satisfies 
$$
\Gamma(\widetilde G(\alpha,\ell)) =\max_{H\subseteq \widetilde G(\alpha,\ell),\ |H|\ge2}\left\lceil\frac{|E(H)|}{\lfloor |H|/2\rfloor}\right\rceil \ge \left\lceil\frac{2\ell |E(G)|}{\ell |G|-1}\right\rceil.
$$ 
If this lower bound exceeds $\Delta(G)$, then $\chi'(\widetilde G(\alpha,\ell))>\Delta(G)$ for every $\Z_\ell$-voltage assignment $\alpha$.
Odd cycles show that this can occur when $\ell$ is odd. Therefore the assumption that $\ell$ is even in Theorem~\ref{thm:edge-coloring-brief} cannot be removed for general graphs.

A graph $G$ is \emph{critical} if $\chi'(H)<\chi'(G)$ for every proper subgraph $H$ of $G$. For such graphs, we determine exactly which values of $s$ in Question~\ref{q1} are attainable when $\ell$ is odd.

\begin{theorem}\label{thm:s-critical-graph}
Let $\ell \ge 3$ be odd, and let $G$ be a critical graph of order $n$ and size $m$ with $\Delta(G)\ge 3$ and $\chi'(G) \ge \Delta(G)+1$. Then, there exists a $\Z_\ell$-voltage assignment $\alpha$ such that $\chi'(\widetilde G(\alpha,\ell)) = s$ if and only if $s\in \{s_0, s_0+1,\dots, \chi'(G)\}$, where $s_0= \min\left\{\chi'(G)-1, \max\left\{\Delta(G), \left\lceil\frac{2\ell m}{\ell n-1}\right\rceil \right\} \right\}$.
\end{theorem}

Observe that in Theorem~\ref{thm:s-critical-graph}, $\tfrac{2\ell m}{\ell n-1} = \tfrac{2m}{n}\ \tfrac{\ell n}{\ell n -1}$. For fixed $n$, this is asymptotic to the average degree of $G$ as $\ell$ goes to the infinity. Thus when $\ell$ is large enough, $s_0$ is close to $\Delta(G)$.

For chromatic numbers, we give a complete answer to Question~\ref{q2}.
\begin{theorem}\label{thm:vt-any-number-Zell}
Let $\ell \ge 2$ be an integer and let $G$ be a graph. The following hold:
\begin{enumerate}[label=\textup{(\roman*)}]
\item If $\chi(G) \ge 3$, then for every $k \in \{3, 4,\dots,\chi(G)\}$, there exists a $\Z_\ell$-voltage assignment $\alpha$ such that $\chi(\widetilde G(\alpha,\ell))=k$. 
\item There exists a $\Z_\ell$-voltage assignment $\alpha$ such that $\chi(\widetilde G(\alpha, \ell)) = 2$ if and only if either $G$ is bipartite or $\ell$ is even.
\end{enumerate}
\end{theorem}

The remainder of the paper is organized as follows. In the next section, we introduce additional notation and definitions, and prove Theorem~\ref{thm:density}. In Section~\ref{sec:edge}, we prove Theorems~\ref{thm:edge-coloring-brief} and~\ref{thm:s-critical-graph}. In Section~\ref{sec:vertex}, we show Theorem~\ref{thm:vt-any-number-Zell}. The paper concludes with some remarks.

\section{Preliminaries}\label{sec:prelim}
Let $G$ be a graph. For $X \subseteq V(G)$, let $G[X]$ denote the subgraph of $G$ induced by $X$, that is, the subgraph with vertex set $X$ and edge set consisting of all edges of $G$ whose endpoints both lie in $X$. 
For $F \subseteq E(G)$, let $G[F]$ denote the subgraph of $G$ spanned by $F$, that is, the subgraph with edge set $F$ and vertex set consisting of all endpoints of edges in $F$. Let $G-F$ denote the spanning subgraph of $G$ with edge set $E(G)\setminus F$, and write $G-e$ for $G-\{e\}$. 

Let $\ell \ge 1$ be an integer, and let $\alpha$ be a $\Z_\ell$-voltage assignment. A vertex $(v,i)$ of $\widetilde G(\alpha,\ell)$ is called a \emph{lifted vertex} of $v$, and an edge $(e,i)$ of $\widetilde G(\alpha, \ell)$ is called a \emph{lifted edge} of $e$. For each $i \in \Z_\ell$, the set $V_i:= V(G)\times \{i\}$ is called the \emph{$i$-th layer} of $\widetilde G(\alpha,\ell)$.
For $F \subseteq E(G)$, define $\widetilde F :=\{(e,i) \,\mid\, e\in F,\ i\in\Z_\ell\}$. Then, $|\widetilde{F}| = \ell|F|$. If $F=\{e\}$, we write $\widetilde e$ for $\widetilde F$.
For a subgraph $H$ of $G$, we write $\widetilde H:=\widetilde H\left(\alpha|_{E(H)},\ell\right)$ for the $\ell$-fold cyclic cover of $H$ induced by the restriction of $\alpha$ to $E(H)$.
If $\supp(\alpha)=\emptyset$, we call $\alpha$ the \emph{all-zero voltage assignment}, in which case $\widetilde G(\alpha,\ell)$ is the disjoint union of $\ell$ copies of $G$.

\begin{theorem}\label{thm:density}
Let $G$ be a graph. For every integer $\ell \ge 1$ and every $\Z_\ell$-voltage assignment $\alpha$, the cyclic cover $\widetilde G(\alpha,\ell)$ satisfies $\Gamma(\widetilde G(\alpha, \ell))\le \Gamma(G)$.
\end{theorem} 
\begin{proof}
Let $\widetilde G :=\widetilde G(\alpha, \ell)$. Since every subgraph $H \subseteq \widetilde G$ satisfies $|E(H)| \le |E(\widetilde G[V(H)])|$, it suffices to show that 
$$|E(\widetilde G[W])| \le \Gamma(G) \cdot \left\lfloor |W|/2 \right\rfloor$$ 
for every subset $W \subseteq V(\widetilde G)$ with $|W|\ge 2$.

Fix such a subset $W \subseteq V(\widetilde G)$. For each $v\in V(G)$, define $\widetilde V_v:=\{v\}\times\Z_\ell$ and $n_v:=|\widetilde V_v \cap W|$. For each $r \in [\ell]:=\{1,2,\dots,\ell\}$, let $U_r:= \{v \in V(G)\,\mid\, n_v\ge r\}$.
Then, $n_v = |\{r\in [\ell] \mid v \in U_r\}|$ for every $v \in V(G)$. Thus 
$$
|W| = \sum_{v \in V(G)}n_v =\sum_{v\in V(G)}\sum_{r\in [\ell]}|\{v\}\cap U_r| =\sum_{r\in[\ell]}|U_r|.
$$

Let $e\in E(G)$ have endpoints $u$ and $v$. The number of edges in $\tilde e$ that belongs to $\widetilde G[W]$ is at most $\min\left\{|\widetilde V_u \cap W|, |\widetilde V_v \cap W|\right\} = \min\{n_u, n_v\}$. Moreover, $$\min\{n_u, n_v\} = |\{r \in [\ell]\mid u,v \in U_r\}| = |\{r \in [\ell]\mid e \in E(G[U_r])\}|.$$
Therefore, summing over all edges of $G$, we obtain
\[
\begin{aligned}
|E(\widetilde G[W])|
&=\sum_{e\in E(G)}
  |\widetilde e\cap E(\widetilde G[W])|\\
&\le \sum_{e\in E(G)}
  |\{r\in[\ell] \mid e\in E(G[U_r])\}|\\
&= \sum_{e \in E(G)} \sum_{r\in [\ell]}|\{e\} \cap E(G[U_r])|\\
&=\sum_{r\in[\ell]}|E(G[U_r])|.
\end{aligned}
\]

For each $r \in [\ell]$, the inequality $|E(G[U_r])| \le \Gamma(G)\lfloor|U_r|/2\rfloor$ holds trivially when $|U_r|\le 1$, and follows from the definition of $\Gamma(G)$ when $|U_r|\ge 2$.
Hence,
$$
|E(\widetilde G[W])| \le \sum_{r\in [\ell]}|E(G[U_r])| \le \Gamma(G)\sum_{r\in[\ell]}\left\lfloor\frac{|U_r|}{2}\right\rfloor \le \Gamma(G) \left\lfloor \frac{\sum_{r \in [\ell]}|U_r|}{2} \right\rfloor = \Gamma(G)\left\lfloor\frac{|W|}{2}\right\rfloor.
$$
Therefore, $\tfrac{|E(\widetilde G[W])|}{\left\lfloor |W|/2\right\rfloor} \le \Gamma(G)$ for every $W\subseteq V(\widetilde G)$ with $|W|\ge 2$, and so $\Gamma(\widetilde G)\le \Gamma(G)$.
\end{proof}

\section{Proofs of Theorems~\ref{thm:edge-coloring-brief} and~\ref{thm:s-critical-graph}}\label{sec:edge}

In this section, we study which values can occur as the chromatic index of an $\ell$-fold cyclic cover of a graph $G$. We will use the following standard notation for partial edge-colorings.
Let $G$ be a graph, and let $\varphi$ be a partial $k$-edge-coloring of $G$. For each vertex $v\in V(G)$, let
\[
\varphi(v) := \{\varphi(e)\,\mid\, e \text{ is incident with } v \text{ and is colored under } \varphi\}
\] 
be the set of colors \emph{present} at $v$, and let
$\overline{\varphi}(v) := [k]\setminus \varphi(v)$ be the set of colors \emph{missing} at $v$. 
For two distinct colors $a, b \in [k]$, let $H$ be the spanning subgraph of $G$ with $E(H) = \varphi^{-1}(a) \cup \varphi^{-1}(b)$. Each component of $H$ is a path, possibly trivial, or an even cycle, possibly a $2$-cycle, whose edges are colored alternately with $a$ and $b$. Such a component is called an \emph{$(a,b)$-chain} of $G$ with respect to $\varphi$. Any two $(a,b)$-chains are either identical or disjoint. For a vertex $v \in V(G)$, let $P_v(a,b)$ denote the $(a,b)$-chain containing $v$. A \emph{Kempe change} on $P_v(a,b)$ is the operation of interchanging the colors $a$ and $b$ along this chain. We denote the resulting coloring by $\varphi':= \varphi/P_v(a,b)$. Then $\varphi'$ is still proper. Moreover, if $v$ is an endpoint of $P_v(a,b)$ and $a \in \overline{\varphi}(v)$, then $\overline{\varphi'}(v) = (\overline{\varphi}(v)\setminus\{a\})\cup\{b\}$, while the missing color set of every internal vertex of $P_v(a,b)$ remains unchanged.

To prove Theorem~\ref{thm:edge-coloring-brief}, we introduce the following two lemmas.

\begin{lemma}\label{lem:double-transfer}
Let $G$ be a graph, let $q \ge \Delta(G)$, and let $H$ be an edge-maximal spanning subgraph of $G$ with $\chi'(H)=q$. Let $F = E(G)\setminus E(H)$ and suppose that $F\ne\emptyset$. If $\alpha$ is the $\Z_2$-voltage assignment with $\supp(\alpha) = F$, then $\chi'(\widetilde G(\alpha, 2))=q$. Moreover, $\chi'(\widetilde G(\alpha',2)) > q$ for every $\Z_2$-voltage assignment
$\alpha'$ with $\supp(\alpha') \subsetneq F$.
\end{lemma}

\begin{proof}
Let $\widetilde G := \widetilde G(\alpha, 2)$. For each $i \in \Z_2$, let $V_i := V(G) \times \{i\}$. Since $\supp(\alpha) = F$, each of the two layers induces a copy of $H$, that is, $\widetilde G[V_i] \cong H$. Hence $\chi'(\widetilde G) \ge \chi'(\widetilde G[V_i]) = \chi'(H) = q$. We now show that $\chi'(\widetilde G) \le q$.

For each $e\in F$, its two lifted edges $(e,0)$ and $(e,1)$ join $V_0$ and $V_1$. We call the pair $\{(e,0),(e,1)\}$ the \emph{crossing pair} of $e$. Fix an ordering $F=\{e_1,\ldots,e_f\}$. Let $\varphi_G$ be a partial $q$-edge-coloring of $G$ such that $H$ is properly colored. 
Define $\psi_0$ by coloring both $\widetilde G[V_0]$ and $\widetilde G[V_1]$ according to
$\varphi_G$, that is, for every $e\in E(H)$ and every $i \in \Z_2$, $\psi_0((e, i))=\varphi_G(e)$. Thus the only uncolored edges are the crossing pairs of edges in $F$. 

Starting with $\psi_0$, we inductively construct partial $q$-edge-colorings $\psi_1,\dots,\psi_f$ such that, for each $j\in \{1,2,\dots,f\}$, the crossing pairs of $e_1,e_2,\dots,e_j$ are colored, while those of $e_{j+1},\dots,e_f$ remain uncolored. Throughout the induction, we maintain the following invariant:
for every $v\in V(G)$, 
$$\overline{\psi_j}((v,0)) = \overline{\psi_j}((v,1)),$$
and their common size is at least the number of edges among $e_{j+1},\dots,e_f$ that are incident with $v$. 

The invariant holds for $j=0$. Indeed, for every $v \in V(G)$, by the definition of $\psi_0$, we have 
$\overline{\psi_0}((v,0)) = \overline{\varphi_G}(v) = \overline{\psi_0}((v,1))$. 
Moreover, since $d_H(v) + d_F(v) = d_G(v) \le \Delta(G) \le q$, it follows that $|\overline{\varphi_G}(v)| = q-d_H(v) \ge d_F(v)$.

Now suppose that $0\le j<f$ and that
$\psi_j$ has been constructed. Suppose that $e_{j+1}$ has endpoints $u$ and $v$. Assume, without loss of generality, that $e_{j+1}$ is oriented from $u$ to $v$ and its crossing pair is $\{(u,0)(v,1), (u,1)(v,0)\}$. Since this crossing pair is still uncolored, the invariant allows us to choose $a \in \overline{\psi_j}((u,1))$ and $b \in \overline{\psi_j}((v,1))$. 
Since the restriction of $\psi_j$ to $\widetilde G[V_1]$ is a proper $q$-edge-coloring of a copy of $H$ and $H+e_{j+1}$ is not $q$-edge-colorable, we have $\overline{\psi_j}((u,1)) \cap \overline{\psi_j}((v,1)) = \emptyset$. In particular, $a\in \psi_j((v,1))$, $b \in \psi_j((u,1))$, and $a \ne b$. Let $P:= P_{(u,1)}(a,b)$. We claim that $(v,1)$ belongs to $P$. Otherwise, $(v,1) \notin V(P)$. Perform a Kempe change on $P$. Then $b$ is missing at both $(u,1)$ and $(v,1)$. Hence the restriction of the resulting coloring to $\widetilde G[V_1] \cong H$ can be extended to a proper $q$-edge-coloring of $H+e_{j+1}$ by assigning $b$ to $e_{j+1}$, a contradiction. Thus $P$ is a path with endpoints $(u,1)$ and $(v,1)$.
Perform a Kempe change on $P$. Then $b$ is missing at $(u,1)$ and $a$ is missing at $(v,1)$, while the missing color sets of all internal vertices of $P$ remain same. In particular, $a$ is still missing at $(u,0)$ and $b$ is still missing at $(v,0)$. Hence we may assign color $a$ to $(u,0)(v,1)$ and color $b$ to $(u,1)(v,0)$. Denote the resulting partial coloring by $\psi_{j+1}$. The following hold:
\begin{itemize}
\item
$\overline{\psi_{j+1}}(w)=\overline{\psi_j}(w)$ for each $w\notin\{(u,0),(u,1),(v,0),(v,1)\}$;
\item
$\overline{\psi_{j+1}}((u,0))
=\overline{\psi_{j}}((u,0))\setminus\{a\} = \overline{\psi_{j+1}}((u,1))$;
\item 
$\overline{\psi_{j+1}}((v,0))
=\overline{\psi_{j}}((v,0))\setminus\{b\} = \overline{\psi_{j+1}}((v,1))$.
\end{itemize}
Moreover, at each of $u$ and $v$, the number of incident edges in $F$ whose crossing pairs remain uncolored decreases by one. Hence the invariant is preserved. By induction, $\psi_f$ is a proper $q$-edge-coloring of $\widetilde G$. Thus $\chi'(\widetilde G) \le q$.

For the moreover statement, let $F' \subsetneq F$, and let $\alpha'$ be the $\Z_2$-voltage assignment with $\supp(\alpha')=F'$. Choose $e \in F\setminus F'$. Since every edge of $H+e$ has voltage $0$ under $\alpha'$, each layer of $\widetilde G(\alpha', 2)$ contains a copy of $H+e$. By the edge-maximality of $H$, $H+e$ is not $q$-edge-colorable. Therefore $\chi'(\widetilde G(\alpha', 2)) > q$.
\end{proof}

Recall that every subgroup of the cyclic group $\Z_\ell$ is cyclic. More precisely, if $d$ is a divisor of $\ell$,
\[ 
C_d^{(\ell)} :=\left\langle \frac{\ell}{d}\right\rangle = \left\{0,\frac{\ell}{d},2\frac{\ell}{d},\ldots,(d-1)\frac{\ell}{d}\right\} 
\] 
is the unique subgroup of $\Z_\ell$ of order $d$. 
Its cosets are 
\[ 
s+C_d^{(\ell)}=\left\{s+c \,\mid\, c\in C_d^{(\ell)}\right\}, \quad s=0,1,\ldots,\frac{\ell}{d}-1,
\] 
which form a partition of $\Z_\ell$.

\begin{lemma}\label{lem:subgroup}
Let $\ell \ge 2$ be an integer, let $d$ be a divisor of $\ell$, and let $G$ be a graph.
Suppose that a $\Z_\ell$-voltage assignment $\alpha$ satisfies $\alpha(e) \in C_d^{(\ell)}$ for every $e\in E(G)$. Define a $\Z_d$-voltage assignment $\alpha_d: E(G) \to \Z_d$ by
\[
\alpha_d(e)=r
\quad\Longleftrightarrow\quad
\alpha(e)=r\frac{\ell}{d}.
\]
Then, $\widetilde G(\alpha, \ell)$ is the disjoint union of the induced subgraphs
\[
\widetilde G(\alpha, \ell)[V(G)\times (s+C_d^{(\ell)})],
\quad s=0,1,\ldots,\ell/d-1,
\]
and each of these induced subgraphs is isomorphic to the $d$-fold cyclic cover $\widetilde G(\alpha_d, d)$.
\end{lemma}
\begin{proof}
For $s = 0,1, \dots,\ell/d-1$, let
\[
X_s=V(G)\times (s+C_d^{(\ell)}) = \bigcup_{j \in \Z_d} V(G)\times\left\{s+j\frac{\ell}{d}\right\}.
\]
The sets $X_s$ partition $V(\widetilde G(\alpha, \ell))$. Since every voltage of $\alpha$ lies in $H$ and $\left(s+C_d^{(\ell)}\right) + C_d^{(\ell)} = s+C_d^{(\ell)}$, any lifted edge with one endpoint in $X_s$ has its other endpoint in $X_s$ as well. Hence $\widetilde G(\alpha, \ell)$ is the disjoint union of the induced subgraphs $\widetilde G(\alpha, \ell)[X_s]$. 

Fix an arbitrary $s \in \{0,1,\ldots,\ell/d-1\}$. Define \[ 
\phi_s:V(\widetilde G(\alpha_d, d)) \to X_s \quad\text{by} \quad \phi_s((v,j)) = \left(v,s+j\frac{\ell}{d}\right). 
\] 
Clearly, $\phi_s$ is a bijection on vertices. Let $e \in E(G)$ be any edge and suppose that $\alpha(e)=r\ell/d$. Then, $\alpha_d(e)=r$. In $\widetilde G(\alpha_d, d)$, the edge $(e,j)$ joins 
\[ 
(t(e),j) \quad\text{and}\quad (h(e),j+r), 
\] 
while in $\widetilde G(\alpha, \ell)[X_s]$, the edge $(e, s+j\ell/d)$ joins 
\[ 
\left(t(e),s+j\frac{\ell}{d}\right) \quad\text{and}\quad \left(h(e),s+(j+r)\frac{\ell}{d}\right). 
\] 
Thus the lifted edge $(e,j)$ of $\widetilde G(\alpha_d, d)$ corresponds exactly to the lifted edge $(e, s+j(\ell/d))$ of $\widetilde G(\alpha, \ell)[X_s]$. Hence $\phi_s$ preserves adjacency and edge multiplicities, and so $\widetilde G(\alpha, \ell)[X_s] \cong \widetilde G(\alpha_d, d)$.
\end{proof}

\begin{proof}[Proof of Theorem~\ref{thm:edge-coloring-brief}.]
We first consider the case $\ell = 2$. If $s=\chi'(G)$, let $\beta$ be the all-zero voltage assignment, and so $\widetilde G(\beta, 2)$ is the disjoint union of two copies of $G$. Hence $\chi'(\widetilde G(\beta,2)) = \chi'(G)$. We may therefore assume that $\chi'(G) \ge \Delta(G)+1$ and $\Delta(G) \le s < \chi'(G)$.  
Fix a proper $\chi'(G)$-edge-coloring of $G$ with color classes $M_1,M_2,\dots,M_{\chi'(G)}$. Let $H_0$ be the spanning subgraph of $G$ with $E(H_0)=M_1\dot\cup M_2\dot\cup\cdots\dot\cup M_s$. Then $\chi'(H_0)=s$, since otherwise an edge-coloring of $H_0$ with fewer than $s$ colors, together with the remaining color classes, would yield an edge-coloring of $G$ with fewer than $\chi'(G)$ colors. Among all spanning subgraphs $H_1$ of $G$ such that $H_0\subseteq H_1$ and $\chi'(H_1) = s$, choose $H$ so that
$|E(H)|$ is maximum. Since $s < \chi'(G)$, we have $H\ne G$, and by the choice of $H$, $\chi'(H') > s$ for every subgraph $H'$ of $G$ with $H'\supsetneq H$. Let $F=E(G)\setminus E(H)$. Then $F \ne \emptyset$. Let $\beta$ be the $\Z_2$-voltage assignment with support $F$. Applying Lemma~\ref{lem:double-transfer} with $q=s$, we obtain $\chi'(\widetilde G(\beta, 2))=s$.

Now let $\ell \ge 2$ be even. If $s=\chi'(G)$, let $\beta$ be the all-zero $\Z_2$-voltage assignment. If $s < \chi'(G)$, let $H$, $F$, and $\beta$ be as constructed above. Thus in the latter case, $\supp(\beta) = E(G)\setminus E(H)$, and in either case, $\chi'(\widetilde G(\beta,2))=s$.
Define a $\Z_\ell$-voltage assignment $\alpha$ by
\[
\alpha(e)=
\begin{cases}
\ell/2, & \text{if } \beta(e)=1,\\
0, & \text{if } \beta(e)=0.
\end{cases}
\]
Clearly, $\supp(\alpha) = \supp(\beta)$.
By Lemma~\ref{lem:subgroup} with $d=2$, $\widetilde G(\alpha, \ell)$ is a disjoint union of $\ell/2$ copies of $\widetilde G(\beta, 2)$. Hence $\chi'(\widetilde G(\alpha, \ell))=\chi'(\widetilde G(\beta, 2))=s$.

We next show that $\alpha$ satisfies the claimed support property. Suppose that $s < \chi'(G)$, and let $\alpha'$ be any $\Z_\ell$-voltage assignment with $\supp(\alpha') \subsetneq \supp(\alpha)$. Choose $e\in \supp(\alpha) \setminus \supp(\alpha')$. Since $\supp(\alpha)=\supp(\beta)= E(G)\setminus E(H)$, we have $e \in E(G)\setminus E(H)$. Moreover, every edge of $H+e$ has voltage $0$ under $\alpha'$, so each layer of $\widetilde G(\alpha', \ell)$ contains a copy of $H+e$. Thus $$\chi'(\widetilde G(\alpha', \ell))\ge \chi'(H+e) > s = \chi'(\widetilde G(\alpha, \ell)).$$
\end{proof}

The next example shows that the support of a voltage assignment alone does not determine the chromatic index of the cyclic cover. It also shows that Theorem~\ref{thm:edge-coloring-brief} does not extend to odd values of $\ell$ for general graphs.
\begin{example}\label{ex:support-not-enough}
Let $G$ be an oriented cycle of order $2r+1$. Let $\ell \ge 2$, and let $\alpha: E(G)\to \Z_\ell$ be a voltage assignment.
Set $s = \sum_{e\in E(G)}\alpha(e) \pmod \ell$. It is easy to check that $\widetilde G(\alpha, \ell)$ is the disjoint union of $\gcd(\ell,s)$ cycles, each of length
$\frac{\ell}{\gcd(\ell,s)}(2r+1)$.
Consequently,
\[
\chi'(\widetilde G(\alpha,\ell))=
\begin{cases}
2, & \text{if } \tfrac{\ell}{\gcd(\ell,s)} \text{ is even},\\
3, & \text{if } \tfrac{\ell}{\gcd(\ell,s)} \text{ is odd}.
\end{cases}
\]
For instance, let $\ell = 6$, let $e \in E(G)$, and let $\alpha$ be a $\Z_6$-voltage assignment with $\supp(\alpha)=\{e\}$. Then $\chi'(\widetilde G(\alpha,6)) = 2$ if $\alpha(e)=3$ and $\chi'(\widetilde G(\alpha, 6))=3$ if $\alpha(e)=4$. 
This also implies that when $\ell$ is odd, $\chi'(\widetilde G(\alpha, \ell)) = 3$ for every $\Z_\ell$-voltage assignment $\alpha$.
\end{example}

We now turn to edge-coloring critical graphs and give a complete answer to Question~\ref{q1} when $\ell$ is odd. We begin with two auxiliary lemmas. The first gives a structural property of critical graphs with $\chi'(G)\ge \Delta(G)+2$. A related result can be found in~\cite{CJZ2025}. We include a proof to keep the paper self-contained.
\begin{lemma}\label{lem:structure}
     Let $G$ be a critical graph of order $n$ and size $m$ with $\chi'(G) =\Delta(G)+k$ for some $k \ge 2$. Then, $n$ is odd and $m = (\Delta(G)+k-1)\frac{n-1}{2}+1$. 
\end{lemma}
\begin{proof}
    Since $\chi'(G) \ge \Delta(G)+2$, by confirmed Goldberg-Seymour Conjecture, $\Gamma(G) = \chi'(G)$. By~\cite{SSTF2012GraphEC}, there exists an odd set $S \subseteq V(G)$ with $|S| \ge 3$ such that $\left\lceil \tfrac{2|E(G[S])|}{|S|-1}\right\rceil = \chi'(G)$. We claim that $S = V(G)$. Suppose for a contradiction that $S\subsetneq V(G)$. Since $G$ is connected, there is an edge $e \notin E(G[S])$. Hence $G[S] \subseteq G-e$, and thus $\chi'(G-e) \ge \chi'(G[S]) \ge \chi'(G)$. On the other hand, as $G$ is critical, we have $\chi'(G-e) < \chi'(G)$, a contradiction. Thus $S = V(G)$, so $n$ is odd and $\left\lceil \tfrac{2m}{n-1} \right\rceil = \chi'(G) = \Delta(G)+k$.
    Moreover, by criticality of $G$, $\left\lceil \tfrac{2(m-1)}{n-1}\right\rceil \le \left\lceil \tfrac{2m}{n-1} \right\rceil - 1$. Since $n$ is odd, this inequality forces $\tfrac{2m-2}{n-1}$ to be an integer, which must equal $\Delta(G)+k-1$. Therefore $m = (\Delta(G)+k-1)\tfrac{n-1}{2}+1$.
\end{proof}

\begin{lemma}\label{lem:|F|-set}
Let $r,p,\ell$ be positive integers, where $\ell$ is odd, and let $X=\{0,1,\dots,r-1\}$. There exist subsets $X_0,X_1,\dots,X_{\ell-1} \subseteq X$, each of size $p$, such that $X_i \cap X_{i+1} = \emptyset$ for every $i \in \Z_\ell$, if and only if $\ell(r-2p) \ge r$.
\end{lemma}
\begin{proof}
Write $\ell = 2s+1$. Suppose that $\ell(r-2p) \ge r$, or $rs \ge p\ell$. 
Consider the sequence $$0,2,4,\dots, 2(p\ell-1) \pmod{\ell}.$$ Since $\gcd(2,\ell)=1$, multiplication by $2$ induces a permutation of $\Z_\ell$. Hence for each $q \in \{0,1,\dots,p-1\}$, the block $2q\ell,2(q\ell+1),\dots, 2(q\ell+\ell-1) \pmod{\ell}$ contains every element of $\Z_\ell$ exactly once. Thus every element of $\Z_\ell$ occurs exactly $p$ times in this sequence. Since $rs \ge p\ell$, we may partition the sequence into at most $r$ consecutive blocks, each of size at most $s$. By adding empty blocks if necessary, denote these blocks by $B_0,B_1,\dots,B_{r-1}$. Every nonempty block has the form $B_x=\{2a_x, 2(a_x+1), \dots, 2(a_x+b_x-1)\} \pmod{\ell}$ for some integers $a_x \ge 0$ and $1 \le b_x \le s$.

We claim that the elements in each block $B_x$ are distinct modulo $\ell$. Indeed, if $2(a_x+q)\equiv 2(a_x+q')\pmod\ell$ for some $0\le q < q'<b_x$, then $q'-q \equiv 0 \pmod \ell$, as $\gcd(2,\ell)=1$. This is impossible because $1\le q'-q\le b_x-1 \le s-1 < \ell$. Thus the elements of each $B_x$ are distinct in $\mathbb Z_\ell$.
Let $C_\ell$ be the cycle whose vertex set is $\Z_\ell$, where each $i\in \Z_\ell$ is adjacent to $i+1$, with addition taken modulo $\ell$.
We claim that each $B_x$ is an independent set in $C_\ell$. Otherwise, suppose that two elements of $B_x$ are adjacent in $C_\ell$. Then, for some $0\le q < q' < b_x$, $2(q'-q)\equiv \pm 1\pmod{\ell}$.
However, $2\le 2(q'-q)\le 2b_x-2 \le 2s-2=\ell-3$, so $2(q'-q)$ is congruent to neither $1$ nor $-1$ modulo $\ell$, a contradiction. Thus $B_x$ is independent. 
For each $i\in \Z_\ell$, define $X_i=\{x\in X \mid i\in B_x\}$.
Because every element of $\Z_\ell$ occurs exactly $p$ times in
the original sequence and no block contains a repeated element, we have $|X_i|=p$ for every $i \in \Z_\ell$.
Moreover, if $x\in X_i\cap X_{i+1}$, then both $i$ and $i+1$ belong to
$B_x$, contradicting the fact that $B_x$ is independent in
$C_\ell$. Therefore, $X_i\cap X_{i+1}=\emptyset$ for every $i\in\Z_\ell$, where the indices are taken modulo $\ell$. In particular, $X_{\ell-1}\cap X_0=\emptyset$.

Now suppose that such sets $X_0, X_1, \dots, X_{\ell-1}$ exist. For each $j \in \Z_r$, denote $I_j := \{i \in \Z_\ell \,\mid\, j \in X_i\}$. Since $X_i \cap X_{i+1} = \emptyset$ for all $i$, we have $|I_j| \le \tfrac{\ell-1}{2}$. By double counting, $$\ell p = \sum_{j \in \Z_r}|I_j| \le r\frac{\ell-1}{2}.$$ Thus $r \le \ell(r-2p)$. This completes the proof.
\end{proof}

\begin{remark}
Lemma~\ref{lem:|F|-set} has an equivalent formulation in terms of edge-coloring. 
Let $p$ and $\ell$ be positive integers, where $\ell$ is odd, and let $C_\ell^p$ be the graph obtained from an $\ell$-cycle by replacing each edge with $p$ parallel edges. Then the lemma is equivalent to the statement that $$\chi'(C_\ell^p) = \left\lceil\frac{2\ell p}{\ell-1}\right\rceil.$$

Seymour~\cite{SEYMOUR198182} conjectured that every planar graph $G$ satisfies $\chi'(G) = \max\{\Delta(G),\Gamma(G)\}$, and later proved the conjecture for series-parallel graphs~\cite{seymour1990}. Here, a graph is \emph{series-parallel} if it has no $K_4$-minor. Since the underlying simple graph of $C_\ell^p$ is a cycle and hence
has no $K_4$-minor, $C_\ell^p$ is series-parallel. Moreover, $\Delta(C_\ell^{p})=2p$ and $\Gamma(C_\ell^{p}) = \lceil 2\ell p/(\ell-1)\rceil$.
Consequently, $\chi'(C_\ell^p) = \lceil2\ell p/(\ell-1)\rceil$. Thus Lemma~\ref{lem:|F|-set} can also be deduced from Seymour's result. However, the proof above is direct and does not rely on that result.
\end{remark}

We now prove the following stronger result, which immediately implies Theorem~\ref{thm:s-critical-graph}.

\begin{theorem}\label{thm:critical-graph}
    Let $\ell \ge 2$ be odd. Let $G$ be a critical graph of order $n$ and size $m$ with $\chi'(G) = \Delta + k$, where $\Delta = \Delta(G) \ge 3$ and $k \ge 1$. The following statements hold:
    \begin{enumerate}[label={(\arabic*)}]
        \item For every $e\in E(G)$ and every $Z_\ell$-voltage assignment $\alpha$ with $\supp(\alpha)=\{e\}$, $\chi'(\widetilde G(\alpha,\ell)) = \chi'(G)-1$.
        \item Suppose that $k \ge 2$. For each $t \in \{0,1,\dots,k-2\}$, there exists a $\Z_\ell$-voltage assignment $\alpha$ such that  $\chi'(\widetilde G(\alpha, \ell)) = \Delta+t$ if and only if $\Delta + t \ge \left\lceil\frac{2\ell m}{\ell n-1}\right\rceil$. 
    \end{enumerate}
\end{theorem}

\begin{proof}
We first prove $(1)$. Let $e \in E(G)$, and let $\alpha:E(G) \to \Z_\ell$ be a voltage assignment with $\supp(\alpha)=\{e\}$. Set $a := \alpha(e) \ne 0$.

Let $g :=\gcd(\ell,a)$ and $d := \ell/g$. 
Since $\langle a \rangle = \langle g \rangle$ has order $d$, Lemma~\ref{lem:subgroup} implies that $\widetilde G(\alpha, \ell)$ is the disjoint union of $g$ induced subgraphs, each isomorphic to the $d$-fold cyclic cover of $G$ obtained by assigning voltage $a/g\in\Z_d$ to $e$ and voltage $0$ to every other edge. Since $\gcd(a/g,d)=\gcd(a/g, \ell/g)=1$, multiplication by the inverse of $a/g$ in $\Z_d$ gives a relabeling of the layers under which the voltage $a/g$ on $e$ becomes $1$. Hence the above $d$-fold cyclic cover is isomorphic to the cyclic cover derived from the voltage assignment $\alpha':E(G)\to\Z_d$ with
\[
\alpha'(e)=1
\quad\text{and}\quad
\alpha'(f)=0
\text{ for every } f\in E(G)\setminus\{e\}.
\]
Therefore, it suffices to show that $\widetilde G:=\widetilde G(\alpha', d)$ satisfies $\chi'(\widetilde G) = \Delta+k-1$. Note that since $\ell$ is odd, $d$ is also odd. At this point, the choice of orientation of $e$ does not affect the argument. Indeed, if $e$ is oriented from $u$ to $v$, then the lifted edges of $e$ join $(u,i)$ to $(v,i+1)$. After relabeling the layers by sending $i$ to $-i$, these edges have the form $(u,j+1)(v,j)$, where $j \in \Z_d$. This is exactly the form obtained when $e$ is oriented from $v$ to $u$ with voltage $1$. Hence we may assume that $e$ is oriented from $v$ to $u$. Then, $(e,i)$ joins $(v,i)$ and $(u,i+1)$.

Since $G$ is critical, $\chi'(G) = \Delta+k$, and $k \ge 1$, we have $\chi'(G-e)=\Delta+k-1 \ge \Delta$. Let $\varphi_G: E(G-e) \to [\Delta+k-1]$ be a proper $(\Delta+k-1)$-edge-coloring.
Then, $$|\overline{\varphi_G}(u)| \ge \Delta-d_G(u)+1 \ge 1 \quad \text{and} \quad |\overline{\varphi_G}(v)| \ge \Delta-d_G(v) + 1 \ge 1.$$  Moreover, $\overline{\varphi_G}(u) \cap \overline{\varphi_G}(v) = \emptyset$, for otherwise $e$ could be colored and $G$ would be $(\Delta+k-1)$-edge-colorable. Thus there exist distinct colors $c_1\in \overline{\varphi_G}(v)$ and $c_2 \in \overline{\varphi_G}(u)$.
Since $\Delta+k-1 \ge \Delta \ge 3$, we may find a third color $c_3 \in [\Delta+k-1]\setminus\{c_1,c_2\}$. 

For each $i \in \Z_d$, let $G_i := \widetilde G[V_i]$. By the definition of $\alpha'$, $G_i \cong G-e$. We define a proper $(\Delta+k-1)$-edge-coloring $\varphi_i$ of $G_i$ from
$\varphi_G$ by permuting the colors as follows:
\[
\begin{array}{c|c}
i & \text{permutation of colors}\\ \hline
0 & c_2\leftrightarrow c_3\\
i\in\{1,3,\dots,d-2\} & c_1\leftrightarrow c_2\\
i\in\{2,4,\dots,d-3\} & \text{identity}\\
d-1 & c_1\leftrightarrow c_3
\end{array}
\]
Consequently, the colors missing at $(v,i)$ and
$(u,i)$ are, respectively,
\[
\begin{array}{c|c|c}
i & (v,i) & (u,i)\\ \hline
0 & c_1 & c_3\\
i\in\{1,3,\dots,d-2\} & c_2 & c_1\\
i\in\{2,4,\dots,d-3\} & c_1 & c_2\\
d-1 & c_3 & c_2
\end{array}
\]
We now color the lifted edges of $e$. For each $i\in \{0,1,\dots,d-2\}$, assign color $c_1$ to $(v,i)(u,i+1)$ when $i$ is even and color $c_2$ when $i$ is odd. For the remaining edge $(v,d-1)(u,0)$, assign color $c_3$. Clearly, each assigned color is missing at both endpoints of the corresponding edge. Therefore, together with the colorings $\varphi_i$, we obtain a proper $(\Delta+k-1)$-edge-coloring of $\widetilde G$. Since $\chi'(\widetilde G) \ge \chi'(G-e) = \Delta+k-1$, we have $\chi'(\widetilde G) = \Delta+k-1$. Therefore, $\chi'(\widetilde G(\alpha, \ell)) = \Delta+k-1 = \chi'(G)-1$.

We now prove $(2)$. Denote $q:=\Delta+k$ and $r:= \Delta+t$. 
Suppose that $\Delta + t \ge \left\lceil\frac{2\ell m}{\ell n-1}\right\rceil$.
Let $e \in E(G)$, and let $\varphi_G:E(G-e) \to [q-1]$ be a proper $(q-1)$-edge-coloring. This coloring induces a decomposition
$$E(G-e)=M_1\dot\cup \cdots \dot\cup M_{q-1},$$
where $M_i$ is the color class receiving color $i$.
By Lemma~\ref{lem:structure}, $m-1=(q-1)\tfrac{n-1}{2}$ and $n$ is odd. Since $|M_i| \le (n-1)/2$ for each $i$, each $M_i$ must have size exactly $(n-1)/2$. Thus each $M_i$ is a near-perfect matching of $G$, that is, each $M_i$ misses exactly one vertex of $G$. Let $H$ be the spanning subgraph of $G$ with $E(H)=M_1\dot\cup M_2 \dot\cup \cdots \dot\cup M_r$ and $F = E(G) \setminus E(H)$. Then, $$|E(H)| = r\frac{n-1}{2} \quad \text{and}\quad |F| = m-r\frac{n-1}{2} = (k-t-1)\frac{n-1}{2}+1.$$
Clearly, $\chi'(H) = r$. 

Since each color class is a near-perfect matching, each color in $[r]$ is missing at exactly one vertex of $G$. Thus the sets $\overline{\varphi_G}(v) \cap [r]$ for $v\in V(G)$ form a partition of $[r]$. Moreover, for each $v \in V(G)$, $d_H(v) = r-|\overline{\varphi_G}(v) \cap [r]|$, and hence
$$d_{G[F]}(v) = d_G(v) - d_H(v) = d_G(v) - r + \left|\overline{\varphi_G}(v) \cap [r]\right| \le \Delta -r +\left|\overline{\varphi_G}(v) \cap [r]\right| \le \left|\overline{\varphi_G}(v) \cap [r]\right|.$$
For each $v \in V(G)$, let $E_{G[F]}(v)$ be the set of edges in $F$ incident with $v$. We may choose an injection 
$$\eta_v:E_{G[F]}(v) \to \overline{\varphi_G}(v) \cap [r].$$ 
For each $f \in F$, say from $u$ to $v$, set $a_f = \eta_u(f)$ and $b_f = \eta_v(f)$. Since the sets $\overline{\varphi_G}(v) \cap [r]$ with $v \in V(G)$ are pairwise disjoint and each $\eta_v$ is injective, the sets $\{a_f, b_f\}$, $f\in F$, are pairwise disjoint. 

Since $r= \Delta+t \ge \left\lceil\frac{2\ell m}{\ell n-1}\right\rceil$, we have $\ell(rn-2m) \ge r$. Moreover, $|F| = m-r\tfrac{n-1}{2}$. So $rn-2m = r-2|F|$, and hence $\ell(r-2|F|) \ge r$. Applying Lemma~\ref{lem:|F|-set} with $X=[r]$ and $p = |F|$, we obtain $|F|$-element subsets $X_0, X_1, \dots, X_{\ell-1} \subseteq [r]$ such that $X_i \cap X_{i+1} = \emptyset$ for every $i \in \Z_\ell$, where indices are taken modulo $\ell$. For each $i \in \Z_\ell$, choose a bijection $\theta_i: F \to X_i$, and then define a map $\psi_i: \{a_f \,\mid\, f\in F\} \cup \{b_f \,\mid\, f\in F\} \to X_i \cup X_{i-1}$ by 
$$\psi_i(a_f) = \theta_i(f) \quad \text{and} \quad \psi_i(b_f) =\theta_{i-1}(f),$$ for every $f \in F$. Since $\theta_i, \theta_{i-1}$ are bijections and $\{a_f \,\mid\, f\in F\} \cap \{b_f \,\mid\, f\in F\} = \emptyset = X_i\cap X_{i-1}$, the map $\psi_i$ is a bijection. Notice that for each $i \in \Z_\ell$, $|[r]\setminus\{a_f, b_f \,\mid\, f \in F\}| = r - 2|F| = |[r]\setminus(X_i\cup X_{i-1})|$, and thus $\psi_i$ can be extended to a permutation $\pi_i$ of $[r]$. 
By construction, for every $f \in F$, $\pi_i(a_f) = \theta_i(f)$ and $\pi_{i+1}(b_f) = \theta_i(f)$. Consequently, $\pi_i(a_f) = \pi_{i+1}(b_f)$ for every $f \in F$ and $i \in \Z_\ell$.

Now we define the voltage assignment $\alpha:E(G) \to \Z_\ell$ by 
\[
\alpha(e)=
\begin{cases}
0, & \text{if } e \in E(H),\\
1, & \text{if } e \in F.
\end{cases}
\]
Clearly, $\supp(\alpha) = F$. For each $i \in \Z_\ell$, $\widetilde G(\alpha, \ell)[V_i] \cong H$, and for each $f \in F$, say from $u$ to $v$, the lifted edge $(f,i)$ joins $(u,i)$ and $(v,i+1)$.
Let $\varphi_H$ be the proper $r$-edge-coloring of $H$ obtained by restricting the coloring $\varphi_G$ to $E(H)$. For each $e \in E(G)$ and $i\in \Z_\ell$, define 
\[
\varphi_{\widetilde G}((e,i))=
\begin{cases}
\pi_i(\varphi_H(e)), & \text{if } e \in E(H),\\
\pi_i(a_e) = \theta_i(e) = \pi_{i+1}(b_e), & \text{if } e \in F.
\end{cases}
\]
We claim this defines a proper $r$-edge-coloring of $\widetilde G(\alpha, \ell)$. Suppose not. Then there exist two distinct edges of $\widetilde G(\alpha, \ell)$ incident with some vertex $(v,i)$ that receive the same color. Denote these two edges by $(e_1, j_1)$ and $(e_2, j_2)$. Then, $e_1$ and $e_2$ are both incident with $v$ in $G$. We consider three cases. 

If $e_1, e_2 \in E(H)$, then $j_1 = j_2 = i$. As $\pi_i$ is a permutation of $[r]$, it follows that $\varphi_H(e_1) = \varphi_H(e_2)$, contradicting the fact that $\varphi_H$ is proper. 

If $e_1, e_2 \in F$, then $j_1, j_2 \in \{i-1, i\}$, and $\eta_v(e_1) = \eta_v(e_2)$, as $\pi_i$ is a permutation. Since $\eta_v$ is an injection, we have $e_1 = e_2$, a contradiction. 

In the remaining case, one of $\{e_1, e_2\}$ lies in $E(H)$ and the other lies in $F$. By symmetry, assume that $e_1 \in E(H)$ and $e_2 \in F$. For the same reason, we obtain $\varphi_H(e_1) = \eta_v(e_2)$. However, $\varphi_H(e_1) \in \varphi_G(v)$, while $\eta_v(e_2) \in \overline{\varphi_G}(v)$, a contradiction. 

Thus no two distinct lifted edges receive the same color, so $\varphi_{\widetilde G}$ is proper. Hence $\chi'(\widetilde G(\alpha, \ell)) \le r$. Since each layer contains $H$ and $\chi'(H) = r$, we have $\chi'(\widetilde G(\alpha, \ell)) = r$.

It remains to show the necessity. By Lemma~\ref{lem:structure}, $n$ is odd. Since $\ell$ is also odd, every $\ell$-fold cyclic cover of $G$ has $\ell n$ vertices, which is odd. Hence each matching in the cyclic cover has size at most $(\ell n -1)/2$. If the cyclic cover admits an $(\Delta+t)$-edge-coloring, then 
$$\Delta+t \ge \left\lceil\frac{\ell m}{(\ell n -1)/2}\right\rceil = \left\lceil\frac{2\ell m}{\ell n -1}\right\rceil.$$ 
The proof is now complete.
\end{proof}

\section{Proof of Theorem~\ref{thm:vt-any-number-Zell}}\label{sec:vertex}
In this section, we focus on vertex-coloring and answer Question~\ref{q2}. Unlike edge-coloring, vertex-coloring is unaffected by parallel edges. Hence, for a multigraph $G$, the chromatic number of $G$ depends only on its underlying simple graph. Therefore, throughout this section, we may assume that all graphs are simple and have at least one edge.

For $\ell = 2$, the $2$-fold cyclic cover considered in this paper corresponds to the double covers studied by Waller~\cite{Waller1976}, after choosing labels $(0,1)$ on the two vertices in each fiber. Waller showed that a double cover has chromatic number at most that of the base graph. Our vertex-coloring result extends and refines this observation to voltage assignments with values in $\Z_\ell$, for arbitrary $\ell \ge 2$. We first introduce an observation on how the chromatic number changes when the voltage assignment is changed on a single edge.  

\begin{observation}\label{obs:single-edge-vertex}
Let $\ell \ge 2$ be an integer and let $G$ be a graph. Suppose that $\alpha, \alpha' :E(G) \to \Z_\ell$ differ on exactly one edge. Then,
$$
\left|\chi(\widetilde G(\alpha,\ell)) - \chi(\widetilde G(\alpha', \ell))\right| \le 1.
$$
\end{observation}
\begin{proof}
It suffices to prove the following fact. Let $H$ and $H'$
be graphs on the same vertex set. Suppose that there is a set $U$
that is independent in both $H$ and $H'$, and that every edge in
$E(H)\triangle E(H')$ is incident with a vertex of $U$. Then, $|\chi(H)-\chi(H')|\le 1$.
Indeed, let $\varphi$ be a proper $\chi(H)$-vertex-coloring of $H$. Recolor every vertex of $U$ with a new color. Since $U$ is independent in $H'$, no
edge of $H'$ has both endpoints in $U$. Moreover, every edge of $H'$ with both endpoints outside $U$ is also an edge of $H$, and hence its endpoints receive distinct colors under $\varphi$. Thus the resulting coloring is a proper $(\chi(H)+1)$-vertex-coloring of $H'$. Therefore $\chi(H')\le \chi(H)+1$.
By symmetry, $\chi(H)\le \chi(H')+1$, and the result follows. 

Let $e = uv$ be the unique edge on which $\alpha$ and $\alpha'$ differ. Applying the fact with $U=\{u\}\times \Z_\ell$, we obtain the desired result.
\end{proof}

\begin{proof}[Proof of Theorem~\ref{thm:vt-any-number-Zell}]
Let $k \in \{3,\dots,\chi(G)\}$. Since $\chi(G) \ge k \ge 3$, $G$ is non-bipartite.
Choose a spanning subgraph $H\subseteq G$ with $\chi(H) = 3$. Such $H$ does exist because $\chi(G) \geq 3$ implies that $G$ contains an odd cycle. Taking this cycle together with the remaining vertices as isolated vertices yields a spanning subgraph $H$ with $\chi(H) = 3$. Let $\varphi_G: V(G)\to \Z_3$ be a $3$-vertex-coloring of $G$ whose restriction to $H$ is proper.

Since every cycle of length $\ell$ is $3$-vertex-colorable, there exists a map $\psi:\Z_\ell \to \Z_3$ such that
\[
\psi(i)\ne \psi(i+1) \quad \text{for all } i \in \Z_\ell.
\]

Define $\alpha: E(G) \to \mathbb Z_\ell$ by 
\[
\alpha(uv)=
\begin{cases}
0, & \text{if } \varphi_G(u) \ne \varphi_G(v),\\
1, & \text{if } \varphi_G(u) = \varphi_G(v),
\end{cases}
\]
for every $uv \in E(G)$.
Now define a vertex-coloring $\varphi_{\widetilde G}$ of $\widetilde G := \widetilde G(\alpha, \ell)$ by $\varphi_{\widetilde G}((v,i))=\varphi_G(v)+\psi(i) \pmod 3$ for every $v \in V(G)$ and $i \in \mathbb Z_\ell$. 
We claim that $\varphi_{\widetilde G}$ is a proper $3$-vertex-coloring of $\widetilde G$. Let $(u,i)(v,i+\alpha(uv))$ be an edge of $\widetilde G$. If $\alpha(uv) = 0$, then $\varphi_G(u) \ne \varphi_G(v)$, and hence $$\varphi_{\widetilde G}((u,i)) = \varphi_G(u)+\psi(i) \ne \varphi_G(v)+\psi(i) = \varphi_{\widetilde G}((v,i)).$$ If $\alpha(uv) = 1$, then $\varphi_G(u) = \varphi_G(v)$, and hence $$\varphi_{\widetilde G}((u,i)) = \varphi_G(u)+\psi(i) \ne \varphi_G(v) + \psi(i+1) = \varphi_{\widetilde G}((v,i+1)).$$ Thus $\varphi_{\widetilde G}$ is proper, so $\chi(\widetilde G) \le 3$. On the other hand, since $\varphi_G$ is proper on $H$, every edge of $H$ receives voltage $0$, so each layer contains a copy of $H$. Hence $\chi(\widetilde G) \ge \chi(H) = 3$. Together, we obtain $\chi(\widetilde G) = 3$. 

Let $\alpha_0$ be the all-zero voltage assignment, so that $\widetilde G(\alpha_0, \ell)$ has chromatic number $\chi(G)$.
Order the edges on which $\alpha_0$ and $\alpha$ differ, and change the voltage assignment from $\alpha_0$ to $\alpha$ one edge at a time. By Observation~\ref{obs:single-edge-vertex}, the chromatic number of the corresponding cyclic cover changes by at most one at each step. Since it starts at $\chi(G)$ and ends at $3$, every value in $\{3,\dots,\chi(G)\}$ occurs as $\chi(\widetilde G(\alpha,\ell))$ for some $\Z_\ell$-voltage assignment $\alpha$. This proves $(i)$.

For $(ii)$, if $G$ is bipartite with bipartition $(X,Y)$, then for any voltage assignment $\alpha$, $X\times\Z_\ell$ and $Y\times\Z_\ell$ form a bipartition of $\widetilde G(\alpha, \ell)$. Hence $\chi(\widetilde G(\alpha, \ell)) = 2$. If $\ell$ is even, assigning voltage $1$ to every edge makes every lifted edge join layers of opposite parity. Thus the resulting cyclic cover is bipartite, and hence has chromatic number $2$.
Conversely, suppose that $G$ is non-bipartite and $\ell$ is odd. Then $G$ contains an odd cycle $C$. For any voltage assignment $\alpha$, every vertex of $\widetilde C := \widetilde C\left(\alpha|_{E(C)}, \ell\right)$ has degree $2$, and hence $\widetilde C$ is a disjoint union of cycles. Moreover, $|E(\widetilde C)| = \ell|E(C)|$, which is odd as both $|E(C)|$ and $\ell$ are odd. This implies that $\widetilde C$, and hence $\widetilde G(\alpha, \ell)$, contains an odd cycle. Therefore $\widetilde G(\alpha, \ell)$ is not bipartite, so $\chi(\widetilde G(\alpha, \ell)) > 2$. This completes the proof.
\end{proof}

\section{Concluding remarks}\label{sec:concluding}
We conclude the paper by discussing how edge-coloring classes behave under cyclic covers and proposing a conjecture on the chromatic index of cyclic covers of arbitrary loopless graphs for odd $\ell$.

For simple graphs, following Vizing's Theorem~\cite{Vizing1964}, a graph $G$ is said to be of \emph{class one} if $\chi'(G) = \Delta(G)$ and of \emph{class two} if $\chi'(G) = \Delta(G)+1$.
For multigraphs, the confirmed Goldberg-Seymour Conjecture gives $\chi'(G) \in \{\max\{\Delta(G), \Gamma(G)\}, \max\{\Delta(G)+1, \Gamma(G)\}\}$.
Following~\cite{SSTF2012GraphEC}, we say that a graph $G$ is of the \emph{first class} if $\chi'(G) = \max\{\Delta(G), \Gamma(G)\}$, and of the \emph{second class} if $\chi'(G)=\Delta(G)+1$ and $\Gamma(G) \le \Delta(G)$.

Since the chromatic index of a cyclic cover does not exceed that of its base graph, every $\ell$-fold cyclic cover of a simple graph that is of class one is also of class one. For multigraphs, however, the situation is different: a graph of the first class may have an $\ell$-fold cyclic cover of the second class. To construct such examples, we first state a useful observation on density.

\begin{observation}\label{obs:bridge-density}
Let $r$ be a positive integer. Let $G_1$ and $G_2$ be vertex-disjoint graphs such that $\Delta(G_i) \le r$ and $\Gamma(G_i) \le r$ for each $i=1,2$. Suppose that, for each $i = 1,2$, there exists a vertex $x_i \in V(G_i)$ such that $d_{G_i}(x_i)\le r-1$. Then the graph $G$ obtained from the union of $G_1$ and $G_2$ by adding one edge $x_1x_2$ satisfies $\Gamma(G) \le r$.
\end{observation}


\begin{example}
Let $r\ge 2$, and let $G_1$ and $G_2$ be connected graphs satisfying
\[
\Delta(G_1)=r,\ \Gamma(G_1) = \chi'(G_1)= r+1, \text{ and }
\max\{\Delta(G_2),\Gamma(G_2)\}=r,\  \chi'(G_2)=r+1.
\]
Assume further that $G_1$ has a vertex $x_1$ with $d_{G_1}(x_1) \le r-1$ and that $G_2$ has a vertex $x_2$ with $d_{G_2}(x_2) \le r-1$. Let $G$ be obtained from the disjoint union of $G_1$ and $G_2$ by adding one edge $x_1x_2$. It is easy to see that $\Delta(G)=r$ and $\Gamma(G) = \chi'(G) = r+1$, so $G$ is of the first class.

Now let $\ell = 2$ and define a voltage assignment $\alpha:E(G)\to \Z_2$ by
\[
\alpha(e)=
\begin{cases}
1, & e\in E(G_1),\\
0, & e\notin E(G_1).
\end{cases}
\]
Let $\widetilde G := \widetilde G(\alpha, 2)$. Clearly, $\Delta(\widetilde G) = \Delta(G) = r$.
Moreover, the subgraph $\widetilde G_1$ of $\widetilde G$ is bipartite and has maximum degree $r$, and hence $\Gamma(\widetilde G_1) \le r$. The subgraph $\widetilde G_2$ consists of two disjoint copies of $G_2$. Since $\alpha(x_1x_2)=0$, the lifted edges of $x_1x_2$ are $(x_1,0)(x_2,0)$ and $(x_1,1)(x_2,1)$, which join $\widetilde G_1$ with the two copies of $G_2$, respectively, at vertices of degree at most $r-1$. Applying Observation~\ref{obs:bridge-density} twice, we obtain $\Gamma(\widetilde G)\le r$.
Since $\widetilde G$ contains a copy of $G_2$, we have $r+1 = \chi'(G_2) \le \chi'(\widetilde G) \le \chi'(G)=r+1$. 
Thus $\chi'(\widetilde G)=r+1$.
Therefore $\chi'(\widetilde G) > \max\{\Delta(\widetilde G),\Gamma(\widetilde G)\} =r$, and hence $\widetilde G$ is of the second class. The following is an example with $r = 3$.

\begin{figure}[H]
\centering
\begin{tikzpicture}[
    vertex/.style={circle, draw, fill=white, inner sep=1.6pt},
    degtwo/.style={circle, draw, fill=gray!25, inner sep=1.8pt},
    edge/.style={line width=0.8pt},
    bridge/.style={line width=1.2pt, dashed}
]
\node[vertex] (a1) at (0,1) {};
\node[vertex] (a2) at (0,-1) {};
\node[degtwo] (a3) at (1.6,0) {};
\draw[edge] (a1) to[out=210,in=150] (a2);
\draw[edge] (a1) to[out=-30,in=30] (a2);
\draw[edge] (a1) -- (a3);
\draw[edge] (a2) -- (a3);
\node at (0.65,-2.15) {$G_1$};

\coordinate (O) at (5.2,0);
\foreach \i/\ang in {1/108,2/36,3/-36,4/-108}{
    \node[vertex] (u\i) at ($(O)+(\ang:1.75)$) {};
}
\foreach \i/\ang in {0/180,1/108,2/36,3/-36,4/-108}{
    \ifnum\i=0
        \node[degtwo] (v\i) at ($(O)+(\ang:0.78)$) {};
    \else
        \node[vertex] (v\i) at ($(O)+(\ang:0.78)$) {};
    \fi
}
\draw[edge] (u1) -- (u2) -- (u3) -- (u4);
\foreach \i in {1,2,3,4}{
    \draw[edge] (u\i) -- (v\i);
}
\draw[edge] (v0) -- (v2);
\draw[edge] (v2) -- (v4);
\draw[edge] (v4) -- (v1);
\draw[edge] (v1) -- (v3);
\draw[edge] (v3) -- (v0);
\node at (5.2,-2.15) {$G_2$};
\draw[edge] (a3) -- (v0);
\node at ($(a3)+(0.18, 0.28)$) {$x_1$};
\node at ($(v0)+(-0.18, 0.28)$) {$x_2$};
\end{tikzpicture}
\caption{Example for $r = 3$}
\end{figure}
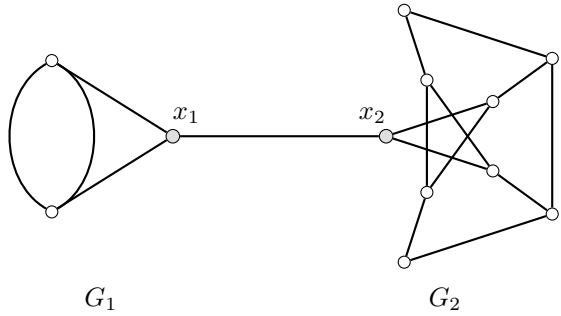
\end{example}

We now consider the chromatic index of cyclic covers for odd values of $\ell$. We first give the following analogue of Observation~\ref{obs:single-edge-vertex} for the chromatic index.
\begin{observation}\label{obs:single-edge}
Let $\ell \ge 2$ be an integer, and let $G$ be a loopless graph. Suppose that $\alpha, \alpha' :E(G) \to \Z_\ell$ differ on exactly one edge. Then,
$$
\left|\chi'(\widetilde G(\alpha, \ell)) - \chi'(\widetilde G(\alpha', \ell))\right| \le 1.
$$
\end{observation}
\begin{proof}
    Suppose that $\alpha$ and $\alpha'$ differ only on an edge $e$, say from $u$ to $v$. 
    Let $$M_\alpha = \{(u,i)(v, i+\alpha(e)) \,\mid\, i\in \Z_\ell)\} \quad\text{and}\quad M_{\alpha'} = \{(u,i)(v, i+\alpha'(e)) \,\mid\, i \in \Z_\ell\}.$$
    Both $M_\alpha$ and $M_{\alpha'}$ are matchings, and after deleting the lifted edges corresponding to $e$, the two graphs $\widetilde G(\alpha, \ell)-M_\alpha$ and $\widetilde G(\alpha', \ell)-M_{\alpha'}$ are identical. 
    Let $p:=\chi'(\widetilde G(\alpha, \ell))$. Then, $\widetilde G(\alpha, \ell)-M_\alpha$ admits a proper $p$-edge-coloring. 
    Now add the edges of $M_{\alpha'}$ and assign all of them one new color. Since $M_{\alpha'}$ is a matching, this gives a proper $(p+1)$-edge-coloring of $\widetilde G(\alpha', \ell)$. Hence $\chi'(\widetilde G(\alpha', \ell)) \le \chi'(\widetilde G(\alpha, \ell))+1$. Interchanging the roles of $\alpha$ and $\alpha'$ gives $\chi'(\widetilde G(\alpha, \ell)) \le \chi'(\widetilde G(\alpha', \ell)) + 1$. The result follows.
\end{proof}
This yields the following corollary.
\begin{corollary}
    Let $\ell \ge 3$ be odd and let $G$ be a loopless graph. Define $$s_\ell(G) := \min_{\alpha:E(G) \to \Z_\ell} \chi'(\widetilde G(\alpha, \ell)).$$ Then 
    $\{\chi'(\widetilde G(\alpha,\ell))\,\mid\,\alpha:E(G)\to \Z_\ell\} =\{s_\ell(G), s_\ell(G)+1,\dots, \chi'(G)\}$.
\end{corollary}
When $G$ is a critical graph with $\Delta(G) \ge 3$ and $\chi'(G) \ge \Delta(G) + 1$, Theorem~\ref{thm:s-critical-graph} gives 
$$
s_\ell(G) = \min\left\{\chi'(G)-1,\max\left\{\Delta(G),\left\lceil \frac{2\ell |E(G)|}{\ell |G|-1}\right\rceil\right\}\right\}.
$$
It is natural to ask whether $s_\ell(G)$ admits a similar formula for arbitrary loopless graphs. To this end, define
$$L_\ell(G):= \max\left\{\Delta(G), \max_{\substack{S\subseteq V(G)\\ |S|\ge 3\text{ odd}}} \left\lceil \frac{2\ell |E(G[S])|}{\ell |S|-1} \right\rceil \right\}.$$
For every voltage assignment $\alpha:E(G)\to \Z_\ell$, we have $\chi'(\widetilde G(\alpha, \ell))\ge L_\ell(G)$. Indeed, if $S\subseteq V(G)$ is odd, then $\widetilde S:=S\times \Z_\ell$ has odd size $\ell|S|$, and $|E(\widetilde G(\alpha, \ell)[\widetilde S])|=\ell|E(G[S])|$. Since each color class contains at most $(\ell|S|-1)/2$ edges with both endpoints in $\widetilde S$, every edge-coloring of $\widetilde G(\alpha, \ell)$ requires at least 
$\left\lceil \frac{2\ell|E(G[S])|}{\ell|S|-1} \right\rceil$ colors. Together with $\chi'(\widetilde G(\alpha, \ell))\ge \Delta(G)$, this gives $\chi'(\widetilde G(\alpha, \ell))\ge L_\ell(G)$.

This leads to the following conjecture.

\begin{conjecture}\label{conj:odd-density} 
Let $\ell\ge 3$ be odd and let $G$ be a loopless graph. Then, $s_\ell(G)=L_\ell(G)$.
\end{conjecture}

\section*{Acknowledgments}
Rong Luo is partially supported by a grant from  Simons Foundation (No. 839830).

\section*{Conflict of Interest Statement}
The authors declare that they have no conflict of interest regarding the publication of this paper.

\bibliographystyle{plain}
\bibliography{Reference.bib}

@article{plachta2020,
  title={Coverings of Cubic Graphs and 3-Edge Colorability},
  author={Plachta, L.},
  journal={Discussiones Mathematicae Graph Theory},
  volume={41},
  number={1},
  pages={311--334},
  year={2020},
  publisher={Uniwersytet Zielonog{\'o}rski. Wydzia{\l} Matematyki, Informatyki i Ekonometrii}
}

@article{PISANSKI1983,
title = {Edge-colorability of graph bundles},
journal = {Journal of Combinatorial Theory, Series B},
volume = {35},
number = {1},
pages = {12-19},
year = {1983},
issn = {0095-8956},
doi = {https://doi.org/10.1016/0095-8956(83)90076-X},
url = {https://www.sciencedirect.com/science/article/pii/009589568390076X},
author = {Pisanski, T. and Shawe-Taylor, J. and Vrabec, J.},
}

@article{ALM2002,
author = {Amit, A. and Linial, N. and Matou{\v{s}}ek, J.},
title = {Random lifts of graphs: Independence and chromatic number},
journal = {Random Structures \& Algorithms},
volume = {20},
number = {1},
pages = {1-22},
doi = {https://doi.org/10.1002/rsa.10003},
url = {https://onlinelibrary.wiley.com/doi/abs/10.1002/rsa.10003},
eprint = {https://onlinelibrary.wiley.com/doi/pdf/10.1002/rsa.10003},
year = {2002}
}

@article{KIM2008,
title = {The chromatic numbers of double coverings of a graph},
journal = {Discrete Mathematics},
volume = {308},
number = {22},
pages = {5078-5086},
year = {2008},
issn = {0012-365X},
doi = {https://doi.org/10.1016/j.disc.2007.09.024},
url = {https://www.sciencedirect.com/science/article/pii/S0012365X07007509},
author = {D. Kim and J. Lee},
}

@article{SEYMOUR198182,
title = {On Tutte's extension of the four-colour problem},
journal = {Journal of Combinatorial Theory, Series B},
volume = {31},
number = {1},
pages = {82-94},
year = {1981},
issn = {0095-8956},
doi = {https://doi.org/10.1016/S0095-8956(81)80013-5},
url = {https://www.sciencedirect.com/science/article/pii/S0095895681800135},
author = {P. D. Seymour},
}

@article{seymour1990,
  title={Colouring series-parallel graphs},
  author={P. D. Seymour},
  journal={Combinatorica},
  volume={10},
  number={4},
  pages={379--392},
  year={1990},
  publisher={Springer}
}

@article{Gross1974,
  author  = {Gross, J. L.},
  title   = {Voltage graphs},
  journal = {Discrete Mathematics},
  volume  = {9},
  number  = {3},
  pages   = {239--246},
  year    = {1974},
  doi     = {10.1016/0012-365X(74)90006-5}
}

@article{GrossTucker1977,
  author  = {Gross, J. L. and Tucker, T. W.},
  title   = {Generating all graph coverings by permutation voltage assignments},
  journal = {Discrete Mathematics},
  volume  = {18},
  number  = {3},
  pages   = {273--283},
  year    = {1977},
  doi     = {10.1016/0012-365X(77)90131-5}
}

@book{GrossTucker1987,
  author    = {Gross, J. L. and Tucker, T. W.},
  title     = {Topological Graph Theory},
  publisher = {Wiley-Interscience},
  address   = {New York},
  year      = {1987}
}

@article{Waller1976,
  author  = {Waller, D. A.},
  title   = {Double covers of graphs},
  journal = {Bulletin of the Australian Mathematical Society},
  volume  = {14},
  number  = {2},
  pages   = {233--248},
  year    = {1976},
  doi     = {10.1017/S0004972700025053}
}

@article{Hofmeister1988,
  author  = {Hofmeister, M.},
  title   = {Counting double covers of graphs},
  journal = {Journal of Graph Theory},
  volume  = {12},
  number  = {3},
  pages   = {437--444},
  year    = {1988},
  doi     = {10.1002/jgt.3190120316}
}

@article{FengKutnarMalnicMarusic2008,
  author  = {Feng, Y. and Kutnar, K. and Malni{\v c}, A. and Maru{\v s}i{\v c}, D.},
  title   = {On 2-fold covers of graphs},
  journal = {Journal of Combinatorial Theory, Series B},
  volume  = {98},
  number  = {2},
  pages   = {324--341},
  year    = {2008},
  doi     = {10.1016/j.jctb.2007.07.001}
}

@article{Vizing1964,
  author  = {Vizing, V. G.},
  title   = {On an estimate of the chromatic class of a {$p$}-graph},
  journal = {Diskret. Analiz},
  volume  = {3},
  pages   = {25--30},
  year    = {1964},
  note    = {In Russian}
}

@article{Goldberg1973,
  author  = {Goldberg, M. K.},
  title   = {On multigraphs of almost maximal chromatic class},
  journal = {Diskret. Analiz},
  volume  = {23},
  pages   = {3--7},
  year    = {1973},
  note    = {In Russian}
}

@book{SSTF2012GraphEC,
  author    = {Stiebitz, M. and Scheide, D. and Toft, B. and Favrholdt, L. M.},
  title     = {Graph Edge Coloring: {Vizing's} Theorem and {Goldberg's} Conjecture},
  publisher = {John Wiley \& Sons},
  address   = {Hoboken, NJ},
  year      = {2012},
  series    = {Wiley Series in Discrete Mathematics and Optimization}
}

@misc{CHYZ2024,
      title={A short proof of the {Goldberg--Seymour} conjecture}, 
      author={G. Chen and Y. Hao and X. Yu and W. Zang},
      year={2024},
      eprint={2407.09403},
      archivePrefix={arXiv},
      primaryClass={math.CO},
      url={https://arxiv.org/abs/2407.09403}, 
      note = {\href{https://arxiv.org/abs/2407.09403}{arXiv:2407.09403}}
}

@article{CJZ2025,
  title={Proof of the {Goldberg--Seymour} conjecture on edge-colorings of multigraphs},
  author={Chen, G. and Jing, G. and Zang, W.},
  journal={Journal of Combinatorial Optimization},
  volume={50},
  number={3},
  pages={23},
  year={2025},
  publisher={Springer}
}

@article{Seymour1979,
    author = {P. D. Seymour},
    title = {On multi-colourings of cubic graphs, and conjectures of {Fulkerson} and {Tutte}},
    journal = {Proceedings of the London Mathematical Society},
    volume = {s3-38},
    number = {3},
    pages = {423--460},
    year = {1979},
    month = {05},
    issn = {0024-6115},
    doi = {10.1112/plms/s3-38.3.423},
    url = {https://doi.org/10.1112/plms/s3-38.3.423},
    eprint = {https://academic.oup.com/plms/article-pdf/s3-38/3/423/4351892/s3-38-3-423.pdf},
}

@misc{J2026,
      title={On Edge Coloring of Multigraphs}, 
      author={G. Jing},
      year={2026},
      eprint={2308.15588},
      archivePrefix={arXiv},
      primaryClass={math.CO},
      url={https://arxiv.org/abs/2308.15588}, 
      note= {arXiv: 2308.15588}
}

\end{document}